\documentclass[12pt]{article}
\usepackage[utf8]{inputenc}
\usepackage[english]{babel}
\usepackage[dvips]{graphicx}
\usepackage{indentfirst}
\usepackage{latexsym}
\usepackage{amsthm}
\usepackage{amsfonts}
\usepackage{amssymb}
\usepackage{amsmath}
\usepackage{alphabeta}
\usepackage{euscript,bm}
\usepackage{array}
\usepackage{epsf,wrapfig}

\usepackage{hyperref}
\hypersetup{
    colorlinks=true,
    linkcolor=blue,
    filecolor=magenta,
    urlcolor=cyan,
    citecolor=green,
}

\theoremstyle{plain}
\newtheorem{theorem}{Theorem}

\newtheorem{proposition}{Proposition}

\theoremstyle{definition}

\newtheorem{property}{Property}

\begin{document}

\noindent UDC 519.17

\begin{center}
{\bf Splitting a graph by a given partition of the set of vertices based on the minimum weight of the induced trees} 
\end{center}
\begin{center}
{\large V.\,A. Buslov}
\end{center}

\begin{abstract}
A method for considering a weighted directed graph with an accuracy of up to a given partition of the set of vertices is proposed. The resulting digraph (the splitting graph) does not contain arcs inside each partition element, and the arcs between the partition atoms are calculated in a special way taking into account the arcs of the original directed graph inside the atoms. This accounting is based on minimal trees defined on atoms. A study was made of what information about the original  digraph is preserved in its splitting.
\end{abstract} 

This paper is based on the works \cite{V6,V7,V8} and also uses the results of \cite{V9,V11}. The definitions and notations correspond to those adopted in them, up to the notation of the original graph and the replacement of some fonts. 

All the results used are presented in the form of a list of properties.  

\section{Designations and definitions} 
 
The set of vertices of the directed graph $G$ is denoted by ${\tt V}G$, the set of its arcs is denoted by ${\tt A}G$. 

The basic object is the weighted directed graph $\it\Psi$, the set of vertices of which for convenience is denoted separately by ${\tt N}$, $|{\tt N}|=N$; the arcs $(i,j)\in{\tt A}\it\Psi$ are assigned real weights $\psi_{ij}$.  

Since the first part of the work only deals with directed graphs, we use the term graph for them until the advent of undirected graphs. 

We study spanning subgraphs (subgraphs with a set of vertices ${\tt N}$) of the directed graph $\it\Psi$, which are entering forests. An entering forest (hereinafter, in the first part of the work, simply a forest) is an directed graph in which no more than one arc emanates from each vertex and there are no contours in it. Connected components of a forest are trees. The only vertex of a tree from which an arc does not emanate is the root (also called the drain). 
The set of roots of a forest $F$ is denoted by ${\tt K}_F$.

For a subgraph $G$ of the graph $\it\Psi$ and a set ${\tt D}\subseteq {\tt N}$, we define (outgoing) weights  
\begin{equation}
\it\Upsilon^G_{\tt D}=\sum_{\begin{smallmatrix}i\in{\tt D} \\ (i,j)\in {\tt A}G\end{smallmatrix}} \psi_{ij} \ , \ \ \it\Upsilon^G=\it\Upsilon^G_{\tt N}=\sum_{(i,j)\in {\tt A}G} \psi_{ij} \ .
\label{ves} 
\end{equation}
Note that we specifically define the weight on the subset ${\tt D}$ of the vertex set ${\tt N}$ for directed graphs so that it also takes into account arcs whose entries do not belong to this subset (their outcomes do). Such a weight is natural for entering forests. For outgoing forests, it is the other way around: when determining the weight, in addition to the internal arcs of the subset ${\tt D}$, one must also take into account arcs whose outcomes do not belong to the set ${\tt D}$ (the entries do). The reason is that in an entering forest, no more than one arc comes from any vertex. Therefore, in entering forests, it is the arcs' outputs that should be monitored, whereas in outgoing ("classical" \ directed) forests, where no more than one arc enters each vertex, the arcs' inputs become important. With this type of input weight, there is a property of its additivity for an arbitrary directed graph $H$: 

\begin{equation}
\it\Upsilon^H_{\tt D\cup\tt D'}=\it\Upsilon^H_{\tt D}+ \it\Upsilon^H_{\tt D'}, \ {\tt D}\cap{\tt D'}=\emptyset, \ \ {\tt D},{\tt D'}\subset {\tt V}H\ .  
\end{equation}

${\cal F}^k$ is a set of spanning forests consisting of $k$ trees.
The minimum weight of $k$-component forests is denoted by 
$$\varphi^k=\min_{F\in{\cal
F}^k}\it\Upsilon^F .$$
If ${\cal F}^k=\emptyset$, we set $\varphi^k=\infty$. In particular, $\varphi^0=\infty$.

$F\in\tilde{\cal F}^k$ means that $F\in{\cal F}^k$ and $\it\Upsilon^F=\varphi^k$. We will agree to call such forests minimal. 

$H=G|_{\tt D}$ --- a subgraph of $G$ induced by the set ${\tt D}$, that is, ${\tt V}H={\tt D}$ and ${\tt A}H$ consists of all arcs of $G$, both ends of which belong to the set ${\tt D}$. $G|_{\tt D}$ is also called the restriction of $G$ to the set ${\tt D}$;

$G^F_{\uparrow{\tt D}}$ --- a graph obtained from $G$ by replacing the arcs emanating from the vertices of the set ${\tt D}$ with arcs emanating from the same vertices in the graph $F$;

If there is an arc whose outgoing point belongs to the set ${\tt D}$, but whose incoming point does not, then we say that the arc comes from the set ${\tt D}$. Similarly, if there is an arc whose incoming point belongs to ${\tt D}$, but whose outgoing point does not, then we say that the arc comes into ${\tt D}$. The outgoing neighborhood ${\tt N}^{out}_{\tt D}(F)$ of the set ${\tt D}$ is the set of entries of arcs outgoing in the graph $F$ from the set ${\tt D}$; the incoming neighborhood ${\tt N}^{in}_{\tt D}(F)$ is defined similarly.

For any subset ${\tt D}\subset {\tt N}$, its complement $\overline{\tt D}={\tt N}\setminus {\tt D}$.

A family \footnote{We use the term family along with the term set to avoid saying "set of sets".} $\mathfrak B $ of nonempty sets ${\tt B}_i$ is called a partition of ${\tt N}$ if ${\tt B}_i\cap{\tt B}_j=\emptyset$ for $i\neq j$, and ${\tt N}=\underset{i}{\cup}{\tt B}_i$.

Let $\mathfrak{B}$ be a family of subsets of ${\tt N}$. The family $\mathfrak A$ consisting of all possible complements and intersections of these subsets is called the algebra {\it generated} by the family $\mathfrak{B}$.

A non-empty set ${\tt A}$ of an algebra $\mathfrak{A}$ is called an {\it atom} if for any element ${\tt B}$ of $\mathfrak{A}$ either ${\tt A}\cap{\tt B}=\emptyset$ or ${\tt A}\cap{\tt B}={\tt A}$.

$\mathfrak{B}_k= \{ {\tt V}T^F_i | i\in{\tt K}_F,\ F\in\tilde{\cal F}^k \}$ --- the family of vertex sets of trees of $k$-component spanning minimal forests; $\mathfrak{A}_k$ --- the algebra of subsets generated by this family; $\aleph_k$ --- the family of atoms of this algebra. An atom ${\tt Z}\in\aleph_k$ is called labeled if it contains the root of at least one forest $F\in\tilde{\cal F}^k$. The family of labeled atoms is denoted by $\aleph_k^\bullet$. The remaining atoms are unlabeled, their family is denoted by $\aleph_k^\circ=\aleph_k\setminus \aleph_k^\bullet$.

\subsection{Trees on subsets}

In \cite{V8}, for any subset ${\tt D}$ of the set of all vertices ${\tt N}$, a set of special tree-like minima is defined:

\begin{equation}
 \lambda_{\tt D}^{\bullet q}=\min_{T\in {\cal T}^{\bullet q}_{\tt D}}\it\Upsilon^T, \  \lambda_{\tt D}^\bullet= \min_{q\in{\tt D}} \lambda_{\tt D}^{\bullet q},  
\label{bt}
\end{equation} 
where ${\cal T}^{\bullet q}_{\tt D}$ is a set of trees with vertex set ${\tt D}$ and root at vertex $q\in {\tt D}$.

Let us also select subsets of trees $\tilde{\cal T}^{\bullet q}$ on which the corresponding minima are achieved: $T\in \tilde{\cal T}^{\bullet q}_{\tt D}$ $\Leftrightarrow$ $T\in {\cal T}^{\bullet q}_{\tt D}$ and $\it\Upsilon^T=\lambda_{\tt D}^{\bullet q}$. We also introduce a set ${\cal T}^\bullet_{\tt D}$ of trees with a vertex set ${\tt D}$.

Similarly, we select from it a subset $\tilde{\cal T}^\bullet_{\tt D}$ according to the rule: $T\in \tilde{\cal T}^{\bullet }_{\tt D}$ $\Leftrightarrow$ $T\in {\cal T}^{\bullet }_{\tt D}$ and $\it\Upsilon^T=\lambda_{\tt D}^{\bullet }$. Naturally done
\begin{equation}
\lambda_{\tt D}^\bullet=\min_{T\in {\cal T}^{\bullet}_{\tt D}}\it\Upsilon^T \ , \ \ {\cal T}^\bullet_{\tt D}= 
\underset{q\in{\tt D}}{\cup}{\cal T}^{\bullet q}_{\tt D} \ . 
\label{Taus}
\end{equation}

Now let arcs emanate from all vertices of the set $\tt D$ and they form a special tree. In \cite{V11} the set of trees ${\cal T}^\circ_{\tt D}$ is introduced, according to the rule: $T\in{\cal T}^\circ_{\tt D}$ means that
$|{\tt V}T|=|{\tt D}|+1$ and $T|_{\tt D}\in{\cal T}^\bullet_{\tt D}$. Thus, in the tree $T$, exactly one arc comes from the set ${\tt D}$ itself, and it comes from the root of the generated tree $T|_{\tt D}$.

Note that for different trees $T$ and $T'$ from ${\cal T}^\circ_{\tt D}$, the sets of their vertices, generally speaking, do not coincide, since the root of the tree can be any vertex from $\overline{\tt D}$.

Let's enter the weight
\begin{equation}
\lambda_{\tt D}^\circ =\min_{T\in {\cal T}_{\tt D}^\circ}\it\Upsilon^{T}   , 
\label{c}
\end{equation} 
and also the subset of $\tilde{\cal T}^\circ_{\tt D}$ trees on which this minimum is achieved: $T\in\tilde{\cal T}^\circ_{\tt D}$ means that $T\in{\cal T}^\circ_{\tt D}$ and $\it\Upsilon^T=\lambda^\circ_{\tt D}$.

\subsection{Outgoing restriction}

We will also use another type of induced subgraphs \cite{V12}. By $G|_{\uparrow{\tt S}}$ we mean a graph $H$ whose arc set is all arcs emanating in $G$ from the vertices of ${\tt S}$ (and only these arcs), and whose vertex set is ${\tt S}$ supplemented by the entries of arcs emanating from ${\tt S}$ in $G$: ${\tt V}H={\tt S}\cup{\tt N}^{out}_{\tt S}(G)$. We will call such a graph an {\it outgoing restriction} of the graph $G$ onto the set ${\tt S}$. In this case, the arcs of the graph $G$ that come out from the vertices of the set ${\tt N}^{out}_{\tt S}(G)$ do not fall into the corresponding outgoing restriction even if their entries belong to the set ${\tt S}\cup{\tt N}^{out}_{\tt S}(G)$.

Note that since in any forest $F$ there is at most one arc emanating from any vertex, if $F|_{\tt S}$ is a tree and ${\tt K}_F\cap{\tt S}=\emptyset$, then $F|_{\uparrow{\tt S}}\in{\cal T}^\circ_{\tt S}$.

\subsection{Properties used} 

First of all, we note the following property of the weights of trees in minimal forests.

\begin{property}\cite[Proposition 2]{V8} {\it 
Let $F\in\tilde{\cal F}^k$, $k=1,2,\cdots,N$, and ${\tt D}$ be the set of vertices of any of its trees, and vertex $q\in {\tt D}$ be the root of this tree, then
\begin{equation}
\lambda_{\tt D}^\bullet = \lambda_{\tt D}^{\bullet q} =\it\Upsilon^F_{\tt D}  . 
\label{mlu}
\end{equation}
}
\end{property}
The sets $\tilde{\cal T}_{\tt D}^\circ$ and $\tilde{\cal T}_{\tt D}^{\bullet}$ are the main objects of this paper. 
For the corresponding weights on an arbitrary subset ${\tt D}\subset{\tt N}$ the following relation holds.

\begin{property}\cite[Proposition 1]{V11}{\it Let ${\tt D}\subsetneq {\tt N}$ and ${\cal T}^\circ_{\tt D}\neq\emptyset$, then  
\begin{equation}
\lambda_{\tt D}^\circ =\min_{q\in {\tt D} }\left( \lambda_{\tt D}^{\bullet q} + \min_{r\notin{\tt D}}\psi_{qr}\right) . 
\label{lo}
\end{equation}}
 \end{property}
 
The calculation of the quantities $\lambda_{\tt D}^\circ$, $\lambda_{\tt D}^{\bullet}$ and $\lambda_{\tt D}^{\bullet q}$, as well as the construction of the corresponding minimal trees from $\tilde{\cal T}_{\tt D}^\circ$, $\tilde{\cal T}_{\tt D}^{\bullet}$ and $\tilde{\cal T}_{\tt D}^{\bullet q}$ is carried out using efficient algorithms \cite{V8,V10,V11}.

We will use one simple property of the arc replacement operation \cite[Corollaries 1,2 from Lemma 1]{V6}, which we formulate more generally. 

\begin{property}{\it Let $F$ and $G$ be two forests and ${\tt V}G\subseteq{\tt V}F$, and the set ${\tt D}\subseteq {\tt V}F\cap{\tt V}G$. Then the graph $F^G_{\uparrow{\tt D}}$ is a forest if any of the following holds: 

1) ${\tt N}^{in}_{\tt D}(F)=\emptyset$;  

2) ${\tt N}^{out}_{\tt D}(G)=\emptyset$. 
}
\end{property}

We will assume that the original graph $\it\Psi$ is sufficiently dense, in the sense that there is at least one spanning tree, i.e. the set ${\cal F}^1$ of spanning forests consisting of one tree is not empty. Then the sets ${\cal F}^k$, $k\in \{1, 2, \ldots, N\}$ are also not empty.

\section{Splitting of the digraph} 

Let there be some partition $\aleph$, the elements of which will subsequently turn out to be the vertices of a graph splitting. It is necessary to determine the principle by which the graph splitting is created. It is also necessary to find out what properties the partition $\aleph$ must have in relation to the original directed graph $\it\Psi$, so that this graph splitting can be defined at all. It is also necessary to define the connections between the elements of the partition $\aleph$, the role of which is played by new arcs. The weights of these arcs are also subject to definition and they are responsible for the strength of the connections between the elements of the partition $\aleph$. 
 
\subsection{Trees of type ${\cal T}_{\tt XY}$}

For further constructions, it is natural to split the set of trees ${\cal T}_{\tt X}^\circ$ into subsets in which the root of the tree belongs to different sets of the partition $\aleph$. Note that for $T\in{\cal T}_{\tt X}^\circ$, the set $({\tt V}T\setminus {\tt X})$ consists of one element, namely the root of the tree $T$, which obviously does not belong to the set ${\tt X}$. We introduce the sets ${\cal T}_{\tt XY}$ according to the rule: $T\in{\cal T}_{\tt XY}$ $\Leftrightarrow$ $T\in{\cal T}_{\tt X}^\circ$ and $({\tt V}T\setminus {\tt X})\cap{\tt Y}\subset {\tt Y}$. Precisely this means that in the tree $T$ there is a unique arc $(x,y)$ such that $x\in{\tt X}$ (this vertex is the root of the tree $T|_{\tt X}$) and $y\in{\tt Y}$.

The set ${\cal T}_{\tt X}^\circ$ is represented as a disjoint union  

\begin{equation*}
{\cal T}_{\tt X}^\circ=\underset{{\tt Y}\in\aleph\setminus \{ {\tt X}\}}\cup {\cal T}_{\tt XY} \ . 
\end{equation*}
Naturally, minimal trees are introduced: $T \in \tilde{\cal T}_{\tt XY} $ $\Leftrightarrow$ $T\in{\cal T}_{\tt XY}$ and $\it\Upsilon^T=\underset{T'\in {\cal T}_{\tt XY}}\min \it\Upsilon^{T'}$. The corresponding minimal weight is denoted by $\lambda_{\tt XY}$. According to Property 2 (\ref{lo}), we can write

\begin{equation}
\lambda_{\tt X}^\circ=\min_{{\tt Y}\in\aleph\setminus \{ {\tt X}\}} \lambda_{\tt XY} \ , \ \  \lambda_{\tt XY} = \underset{T\in {\cal T}_{\tt XY}}\min \it\Upsilon^{T} =\min_{q\in {\tt X} }\left(\lambda_{\tt X}^{\bullet q} + \min_{r\in{\tt Y}}\psi_{qr}\right) \ . 
\label{lxy}
\end{equation} 

If, in addition to the graph $\it\Psi$, some other graph $\it\Psi'$ with the same set of nodes is used (for example, some spanning subgraph of the graph $\it\Psi$, in particular, some spanning forest $F$), then we specify it as an argument: ${\cal T}_{\tt XY}(\it\Psi')$. For the graph $\it\Psi$ itself, we omit the reference to it: for example, ${\cal T}_{\tt XY}= {\cal T}_{\tt XY}(\it\Psi)$. The same applies to the weights of sets: say $\lambda_{\tt X}^\bullet (\it\Psi')$ and $\lambda_{\tt X}^\bullet=\lambda_{\tt X}^\bullet (\it\Psi)$.  

\subsection{Digraph splitting based on the tree minimality principle}

Let $\it\Psi$ be an directed graph on the set of vertices {\tt N} and  $\aleph$ be some partition  of the set ${\tt N}$.
We will say that the directed graph $\it\Psi$ is {\it tree-divisible by the partition} $\aleph$, and the partition itself with respect to the directed graph $\it\Psi$ is called a {\it tree-division} if for any ${\tt X}\in\aleph$, the set ${\cal T}_{\tt X}^\bullet\neq \emptyset$.  Only for a graph $\it\Psi$ divisible by partition $\aleph$ we define {\it splitting digraph} by partition $\aleph$ (in what follows, simply splitting digraph, since the same partition appears everywhere in what follows). We will denote it by $\it\Psi^\aleph= \it\Psi|\aleph$. For convenience, we will use both the left and right parts of this expression as notation. The set of vertices of the partition graph $\it\Psi|\aleph$ are the elements of the partition $\aleph$: ${\tt V}\it\Psi^\aleph= \aleph$. An arc $(\tt X,Y)$, $\{{\tt X,Y}\}\subset \aleph$, belongs to the set of arcs of the splitting graph $\it\Psi|\aleph$ if and only if ${\cal T}_{\tt XY}\neq \emptyset$: ${\tt A}\it\Psi^\aleph=\{ ({\tt X,Y})| {\cal T}_{\tt XY}\neq \emptyset \}$. \footnote{We define tree divisibility and the splitting digraph in such a way that it is natural for the entering forests and trees.} 
 
Note that if ${\cal T}_{\tt XY}\neq \emptyset$, then automatically ${\cal T}^\circ_{\tt X}\neq \emptyset$, and also that ${\cal T}^\bullet_{\tt X}\neq \emptyset$. 

We will call this method of graph splitting tree splitting.

The tree divisibility of the directed graph $\it\Psi$ does not mean that its splitting $\it\Psi^\aleph$ has, say, a spanning tree. Moreover, a limiting situation is possible when the set of forests ${\cal F}^k(\it\Psi|\aleph)$ is not empty only for $k=|\aleph|$ (the set ${\cal F}^{|\aleph|}(\it\Psi|\aleph)\neq \emptyset$ by the definition of divisibility), although the direct graph $\it\Psi$ itself has a spanning tree. 
Moreover, it is obvious that if we can divide some spanning subgraph of the graph $\it\Psi$, then we also can divide the graph $\it\Psi$ itself.

If the directed graph $\it\Psi$ is weighted, then the weights of the arcs of the splitting digraph $\it\Psi^\aleph$ will be determined based on the minimality principle. Namely, the weight of any of its arcs $(\tt X,Y)$ is assumed to be equal to 

\begin{equation}
\psi_{\tt XY}^\aleph=\lambda_{\tt XY} -\lambda_{\tt X}^\bullet \ , \ \ {\cal T}_{\tt XY}\neq\emptyset \ .
\label{wei}
\end{equation}
If the set ${\cal T}_{\tt XY}=\emptyset$, then the arc ({\tt X,Y}) is not in the splitting  digraph.

For an unweighted directed graph $\lambda_{\tt XY}=|{\tt X}|$ and $\lambda_{\tt X}^\bullet= |{\tt X}|-1$. Therefore, the weight of any arc according to formula (\ref{wei}) is automatically equal to one.
 
The idea reflected in the introduced weights (\ref{wei}) is the following. The weight of the arc $(\tt X,Y)$ shows how much the weight of the minimal tree $T\in\tilde{\cal T}^\bullet_{\tt X}$ on the vertex set ${\tt X}$ changes when moving to a tree from $\tilde{\cal T}_{\tt XY}$, which has one more arc and it enters ${\tt Y}$.  Taking into account (\ref{lxy}), the expression (\ref{wei}) takes the form

\begin{equation}
\psi_{\tt XY}^\aleph =\min_{q\in {\tt X} }\left(\lambda_{\tt X}^{\bullet q} + \min_{r\in{\tt Y}}\psi_{qr}\right)-\lambda_{\tt X}^\bullet \ . 
\label{vxy}
\end{equation} 

In fact, from the moment of defining the weights, we can forget about the structure of the distribution of arcs inside the partition elements and consider the graph $\it\Psi|\aleph$ as an ordinary weighted digraph, which is what it essentially is. Its vertices are simply the partition $\aleph$ elements.

\begin{figure}[h]
\unitlength=0.7mm
\begin{center}
\begin{picture}(195,50)

\put(24,45){$\it\Psi$}

\put(28,21){\scriptsize $ u$}
\put(27,2){\scriptsize $ v$}
\put(3,2){\scriptsize $ y$}
\put(48,20){\scriptsize $ r$}
\put(47,2){\scriptsize $ q$}
\put(49,38){\scriptsize $ p$}
\put(2,21){\scriptsize $ x$}
\put(2,38){\scriptsize $ s$}
\put(23,38){\scriptsize $ t$}

\put(8,6){\vector(1,0){17}}
\put(25,5){\vector(-1,0){17}}
\put(25,36){\vector(-1,0){17}}
\put(7,8){\vector(0,1){12}}
\put(6,20){\vector(0,-1){12}}
\put(26,20){\vector(0,-1){12}}
\put(27,8){\vector(0,1){12}}
\put(7,35){\vector(0,-1){12}}
\put(8,21){\vector(1,0){17}}
\put(6,23){\vector(0,1){12}}
\put(45,20){\vector(-4,-3){17}}
\put(28,36){\vector(1,0){17}}
\put(27,35){\vector(0,-1){12}}
\put(46,35){\vector(0,-1){12}}
\put(46,8){\vector(0,1){12}}
\put(25,36){\vector(-4,-3){17}}
\put(26,23){\vector(0,1){12}}
\put(25,20){\vector(-4,-3){17}}
\put(47,20){\vector(0,-1){12}}
\put(47,23){\vector(0,1){12}}

\put(5,5){$\bullet$}
\put(5,20){$\bullet$}
\put(5,35){$\bullet$}
\put(25,5){$\bullet$}
\put(25,20){$\bullet$}
\put(25,35){$\bullet$}
\put(45,5){$\bullet$}
\put(45,20){$\bullet$}
\put(45,35){$\bullet$}
%\put(16,6){\oval(24,12)}
%\put(16,29){\oval(27,25)}
%\put(46,21){\oval(10,35)}

\put(17,1){\scriptsize $7$}
\put(17,7){\scriptsize $5$}
\put(14,14){\scriptsize $6$}
\put(15,22){\scriptsize $1$}
\put(14,30){\scriptsize $1$}
\put(14,37){\scriptsize $3$}
\put(7,27){\scriptsize $1$}
\put(8,13){\scriptsize $1$}
\put(2,13){\scriptsize $4$}
\put(23,13){\scriptsize $7$}
\put(2,28){\scriptsize $2$}
\put(35,31){\scriptsize $1$}
\put(34,15){\scriptsize $1$}
\put(42,12){\scriptsize $2$}
\put(42,27){\scriptsize $2$}
\put(22,27){\scriptsize $2$}
\put(48,12){\scriptsize $1$}
\put(48,27){\scriptsize $2$} 
\put(28,27){\scriptsize $3$}
\put(28,13){\scriptsize $4$}

\put(92,45){$\it\Psi|\aleph$}

\put(86,6){\oval(24,12)}
\put(86,29){\oval(27,25)}
\put(116,20){\oval(10,41)}

\put(78,13){\scriptsize $3$}
\put(93,13){\scriptsize $5$}
\put(105,23){\scriptsize $2$}
\put(104,14){\scriptsize $2$}

\put(114,19){\vector(-4,-3){17}}
\put(97,27){\vector(1,0){17}}
\put(92,20){\vector(0,-1){12}}
\put(81,8){\vector(0,1){12}} 

\put(85,27){${\tt X}$}
\put(85,4){${\tt Y}$}
\put(115,22){${\tt Z}$}

\put(154,45){$\{ T\}=\tilde{\cal F}^1$}

\put(146,8){\vector(0,1){12}}
\put(146,35){\vector(0,-1){12}}
\put(168,36){\vector(1,0){17}}
\put(148,21){\vector(1,0){17}}
\put(166,23){\vector(0,1){12}}
\put(185,20){\vector(-4,-3){17}}
\put(186,8){\vector(0,1){12}}
\put(186,35){\vector(0,-1){12}}

\put(145,5){$\bullet$}
\put(145,20){$\bullet$}
\put(145,35){$\bullet$}
\put(165,5){$\bullet$}
\put(165,20){$\bullet$}
\put(165,35){$\bullet$}
\put(185,5){$\bullet$}
\put(185,20){$\bullet$}
\put(185,35){$\bullet$}
\put(156,6){\oval(24,12)}
\put(156,29){\oval(27,25)}
\put(186,21){\oval(10,41)}

\put(175,31){\scriptsize $1$}
\put(155,22){\scriptsize $1$}
\put(163,27){\scriptsize $2$}
\put(147,27){\scriptsize $1$}
\put(147,13){\scriptsize $1$}
%\put(175,22){\scriptsize $1$}
%\put(174,29){\scriptsize $1$}
\put(174,15){\scriptsize $1$}
\put(183,12){\scriptsize $2$}
\put(183,27){\scriptsize $2$}

\end{picture} 
\caption{\small A minimum weight spanning tree $T$ is not divisible by a partition $\aleph$, since its induced subgraph $T|_{\tt Y}$ is not a tree, and therefore $\tilde{\cal T}^\bullet_{\tt Y}(T)=\emptyset$. On calculating the weights of the partition directed graph ${\it\Psi}^\aleph$, for example: $\lambda_{\tt X}^\bullet={\it\psi}_{tx}+{\it\psi}_{sx}+{\it\psi}_{xu}=3$, $\lambda_{\tt XY}={\it\psi}_{ut}+{\it\psi}_{tx} +{\it\psi}_{sx}+{\it\psi}_{xy}=8$, whence ${\it\psi^\aleph_{\tt XY}=\lambda_{\tt XY}}-\lambda_{\tt X}^\bullet=5$. }
\label{p1}
\end{center}
\end{figure}
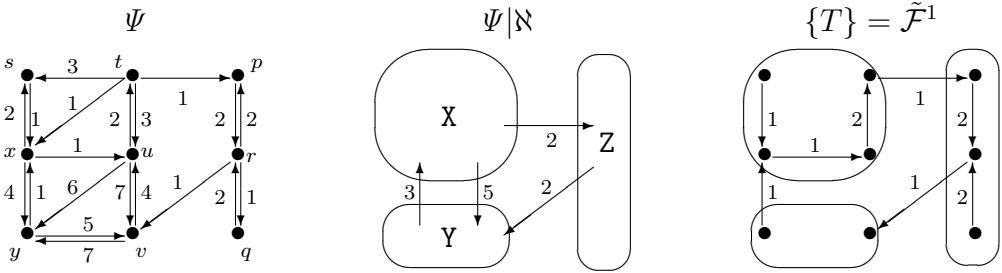

\subsection{Tree-like divisibility of forests and their splitting}

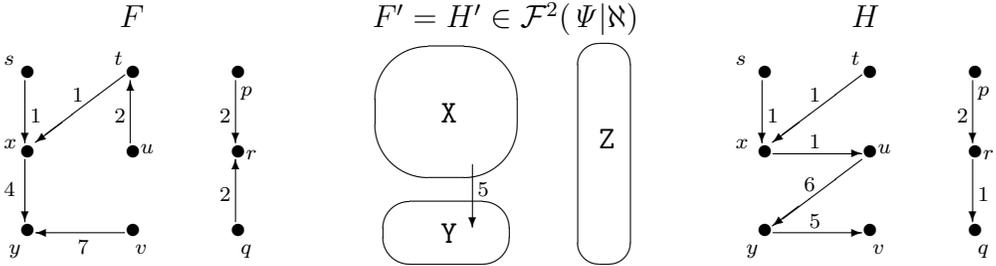
\begin{figure}[h]
\unitlength=0.7mm
\begin{center}
\begin{picture}(195,50)

\put(24,45){$F$}

\put(28,21){\scriptsize $ u$}
\put(27,2){\scriptsize $ v$}
\put(3,2){\scriptsize $ y$}
\put(48,20){\scriptsize $ r$}
\put(47,2){\scriptsize $ q$}
\put(47,32){\scriptsize $ p$}
\put(2,22){\scriptsize $ x$}
\put(2,38){\scriptsize $ s$}
\put(23,38){\scriptsize $ t$}

\put(25,6){\vector(-1,0){17}}
\put(6,20){\vector(0,-1){12}}
\put(6,35){\vector(0,-1){12}}
%\put(25,36){\vector(-1,0){17}}
\put(46,35){\vector(0,-1){12}}
\put(46,8){\vector(0,1){12}}
\put(26,23){\vector(0,1){12}}
\put(25,36){\vector(-4,-3){17}}

\put(5,5){$\bullet$}
\put(5,20){$\bullet$}
\put(5,35){$\bullet$}
\put(25,5){$\bullet$}
\put(25,20){$\bullet$}
\put(25,35){$\bullet$}
\put(45,5){$\bullet$}
\put(45,20){$\bullet$}
\put(45,35){$\bullet$}

\put(16,2){\scriptsize $7$}
\put(2,13){\scriptsize $4$}
\put(15,31){\scriptsize $1$}
\put(7,27){\scriptsize $1$}
\put(43,12){\scriptsize $2$}
\put(43,27){\scriptsize $2$}
\put(23,27){\scriptsize $2$}

\put(72,45){$F'=H'\in{\cal F}^2(\it\Psi|\aleph)$}

\put(86,6){\oval(24,12)}
\put(86,29){\oval(27,25)}
\put(116,21){\oval(10,42)}

\put(92,13){\scriptsize $5$}
\put(91,19){\vector(0,-1){12}} 

\put(85,27){${\tt X}$}
\put(85,4){${\tt Y}$}
\put(115,22){${\tt Z}$}

\put(163,45){$H$}

\put(168,21){\scriptsize $ u$}
\put(167,2){\scriptsize $ v$}
\put(143,2){\scriptsize $ y$}
\put(188,20){\scriptsize $ r$}
\put(187,2){\scriptsize $ q$}
\put(187,32){\scriptsize $ p$}
\put(141,22){\scriptsize $ x$}
\put(141,38){\scriptsize $ s$}
\put(163,38){\scriptsize $ t$}

\put(148,6){\vector(1,0){17}}
\put(146,35){\vector(0,-1){12}}
%\put(165,36){\vector(-1,0){17}}
\put(148,21){\vector(1,0){17}}
\put(186,35){\vector(0,-1){12}}
\put(186,20){\vector(0,-1){12}}
\put(165,20){\vector(-4,-3){17}}
\put(165,36){\vector(-4,-3){17}}

\put(145,5){$\bullet$}
\put(145,20){$\bullet$}
\put(145,35){$\bullet$}
\put(165,5){$\bullet$}
\put(165,20){$\bullet$}
\put(165,35){$\bullet$}
\put(185,5){$\bullet$}
\put(185,20){$\bullet$}
\put(185,35){$\bullet$}

\put(155,7){\scriptsize $5$}
%\put(155,1){\scriptsize 3}
\put(154,14){\scriptsize $6$}
\put(155,22){\scriptsize $1$}
\put(155,31){\scriptsize $1$}
\put(147,27){\scriptsize $1$}
\put(187,12){\scriptsize $1$}
\put(183,27){\scriptsize $2$}

\end{picture} 
\caption{\small  Spanning forests $F$ and $H$ from ${\cal F}^2$ have the same representative $F'=H'$ in ${\cal F}^2(\it\Psi^\aleph)$. Their own partitions $F^\aleph$ and $H^\aleph$ have the same unique arc $({\tt X,Y})$ as $F'$, but with different weights: $f^\aleph_{\tt XY}={\it\psi}_{xy}=4$, $h^\aleph_{\tt XY}={\it\psi}_{uy}=6$. In this case, $f'_{\tt XY}= h'_{\tt XY}={\it\psi}^\aleph_{\tt XY}=5$.}
\label{p2}
\end{center}
\end{figure}

If $F$ is an arbitrary spanning forest of the original graph $\it\Psi$, then it is, generally speaking, not divisible. Let $F$ be such that $F|_{\tt X}$ is a tree, for all ${\tt X}\in\aleph$. That is, this forest is tree-divisible by partition $\aleph$ in the sense of the definition introduced above. The partition $\aleph$ itself with respect to the forest $F$ turns out to be a tree partition \cite{V9}. Only for tree-divisible forests can one define their own splitting $F|\aleph$. The tree divisibility itself imposes restrictions on the number of trees in the spanning forest.

\begin{proposition}
 Let a spanning forest $F$ be tree-divisible by a partition $\aleph$, then any set ${\tt X}\in\aleph$ contains at most one root of the forest $F$.
\end{proposition}

\begin{proof}
Indeed, if at least one set ${\tt X}\in\aleph$ contains more than one root of the forest $F$, then the graph $F|_{\tt X}$ is not a tree. 
\end{proof}

\begin{proposition}
Let $F\in{\cal F}^k$ and it is tree-divisible by partition $\aleph$. Then $k\leq |\aleph|$.
\end{proposition}

\begin{proof}
By contradiction. If $k>|\aleph|$, then there exists at least one set ${\tt X}\in\aleph$ that contains more than one root of the forest $F$, which contradicts Proposition 1.
\end{proof} 

For a tree-divisible forest $F\in{\cal F}^k$, one can obtain its splitting $F|\aleph$ independently of the original graph $\it\Psi$ (from $\it\Psi$ only the weights of its arcs are present in the forest $F$), but simply by the splitting procedure described above.

With respect to a tree-divisible forest $F$, the partition $\aleph$ can be represented as a disjoint union $\aleph^\bullet_F\cup\aleph^\circ_F$, where ${\tt Z}\in\aleph_F^\bullet$ means that ${\tt Z}$ contains a root of forest $F$ (exactly one), and ${\tt X}\in\aleph^\circ_F$ means that ${\tt X}$ does not contain roots of forest $F$. Since any set ${\tt Y}\in\aleph$ contains at most one root of the forest $F$, then $|\aleph_F^\bullet|=|{\cal K}_F|$. Accordingly, $|{\tt A}F|=|\aleph|- |\aleph^\bullet_F|= |\aleph|- |{\cal K}_F| =|\aleph^\circ_F|$. In particular, ${\tt K}_{F^\aleph}=\aleph^\bullet_F$. Therefore, it is true

\begin{proposition}
Let $F\in{\cal F}^k$ be tree-divisible by partition $\aleph$. Then graph $F^\aleph$ consists of $k$ trees and has $|\aleph|-k$ arcs.
\end{proposition}

For the graph $\it\Psi|\aleph$, we can consider its spanning forests, consisting of $k\le |\aleph|$ trees, namely the set ${\cal F}^k(\it\Psi|\aleph)$ (for them, the question of divisibility does not arise at all, since they are subgraphs of an already implemented splitting, as long as ${\cal F}^k(\it\Psi|\aleph)\neq \emptyset$). At the same time, we can consider both tree-divisible forests $F$ from the set ${\cal F}^k$ (if such exist) and their splittings. The question arises: how are these forests related to each other? 
Let $F\in{\cal F}^k$ be tree-divisible by partition $\aleph$. We will call a spanning forest $F'$ of a directed graph $\it\Psi|\aleph$ a {\it representative} of the forest $F$ if ${\tt A}F'={\tt A}F^\aleph$. In this case, in particular, ${\tt K}_{F'}=\aleph^\bullet_F$, and $F'$ automatically belongs to the set ${\cal F}^k(\it\Psi|\aleph)$. The forest $F$ with respect to $F'$ turns out to be a {\it principal}.  

\begin{proposition}
Let $F\in{\cal F}^k$ be tree-divisible by partition $\aleph$. Then it has a unique representative $F'\in{\cal F}^k(\it\Psi|\aleph)$. 

\end{proposition}

\begin{proof}
By the definition of forest divisibility, for any ${\tt X}\in\aleph$ the graph $F|_{\tt X}$ is a tree. Therefore $\{F|_{\tt X}\}={\cal T}^\bullet_{\tt X}(F)\neq \emptyset$. 
Further, since $F$ is tree divisible, then from any set ${\tt X}\in\aleph$ there is no more than one arc emanating from it. And this arc, if it exists, enters the unique set ${\tt Y}\in\aleph\setminus\{ \tt X \}$. Thus, in this case the set ${\cal T}_{\tt XY}(F)=\{ F|_{\uparrow{\tt X} }\}\neq \emptyset$. Therefore, according to the definition of splitting, in the graph $F^\aleph$ there is an arc $({\tt X,Y})$.
Further, since ${\cal T}_{\tt XY}(F)\subset {\cal T}_{\tt XY}$, then the arc $({\tt X,Y})$ is also present in the graph $\it\Psi^\aleph$. If now in the graph $\it\Psi^\aleph$ we take a spanning subgraph $F'$ with all such arcs (and there are exactly $|\aleph|-k$ of them), then it belongs to the set ${\cal F}^k(\it\Psi|\aleph)$ and by construction ${\tt A}F'={\tt A}F^\aleph$.  
\end{proof} 

Thus, a tree-divisible forest is a forest that has a representative. In this sense, a splitting graph is a representative graph. But to call a splitting graph such, the following must hold:

\begin{proposition}
Let $F'\in{\cal F}^k(\it\Psi|\aleph)$, then it has at least one principal $F\in{\cal F}^k$.
\end{proposition}
\begin{proof}
Let ${\tt (X,Y)}\in{\tt A}F'$, then $\tt (X,Y)\in{\tt A}\it\Psi^\aleph$, since $F'$ is a (spanning) subgraph of ${\it\Psi}^\aleph$. By definition, this means, in particular, that ${\cal T}_{\tt XY}\neq \emptyset$. Take any tree $T_{xy}\in{\cal T}_{\tt XY}$ rooted at some vertex $y\in{\tt Y}$ and with a single arc $(x,y)$ entering it. We will do the same with all other arcs of the forest $F'$, selecting any corresponding trees from the sets ${\cal T}_{\tt XY}$. Consider the spanning subgraph $F$ of the graph $\it\Psi$, consisting of the union of these trees, and add to it arbitrary trees $T_z\in{\cal T}^\bullet_{\tt Z}$ with roots at some vertices $z\in{\tt Z}$ for all ${\tt Z}\in{\tt K}_{F'}$:  

\begin{equation}
{\tt A}F=\underset{{\tt (X,Y)}\in{\tt A}F'}{\cup}{\tt A}T_{xy} \underset{{\tt Z}\in{\tt K}_{F'}}{\cup}{\tt A}T_z \ .
\label{ff'}
\end{equation}
This is the required forest $F\in{\cal F}^k$, since by construction ${\tt A}F^\aleph={\tt A}F'$. 
\end{proof}
Of course, the forest $F$ obtained by such a construction is, generally speaking, not unique due to the arbitrariness in the choice of trees $T_{xy}$ from the sets ${\cal T}_{\tt XY}$. Two tree-divisible forests $F$ and $H$ can have the same representative $F'=H'$. In this case, the corresponding forests of splitting $F^\aleph$ and $H^\aleph$ have the same sets of arcs ${\tt A}F^\aleph={\tt A}H^\aleph$. However, the weights of these arcs differ, as the following shows

 \begin{proposition}
Let $F$ be a tree-divisible spanning forest of $\it{\Psi}$ and $F^\aleph$ be a splitting of $F$. Then the weights of their arcs are related by
\begin{equation}
f_{\tt XY}^\aleph=\psi_{xy} \ , \ \ \  (x,y)\in{\tt A}F  , \ \  x\in{\tt X}  , \ \ y\in{\tt Y} \ . 
\end{equation}
\label{fv}
\end{proposition} 
\begin{proof}
Indeed, let the arc $(\tt X,Y)$ be present in the forest $F^\aleph$. This means that ${\cal T}_{\tt XY}(F)= \{ F|_{\uparrow{\tt X} }\}\neq \emptyset$. Thus, in $F$ there is a unique arc $(x,y)$ such that $x\in{\tt X}$ and $y\in{\tt Y}$ with weight $\psi_{xy}$.
Let's see what the weight of the arc $(\tt X,Y)$ in the forest $F^\aleph$ is. According to (\ref{wei}), for the splitting $F^\aleph$ of the forest $F$ we have

\begin{equation}
f_{\tt XY}^\aleph=
\lambda_{\tt XY}(F) -\lambda_{\tt X}^\bullet(F)=\it\Upsilon^F_{\tt X}-\it\Upsilon^{F|_{\tt X}}=\psi_{xy} \ . 
\label{weifa}
\end{equation}
\end{proof}
This is obviously to be expected, since for a forest $F$, the forest $F^\aleph$ is its own splitting. 
The weight of the same arc ({\tt X,Y}) in the forest $F'\in{\cal F}^k(\it\Psi^\aleph)$, which is a representative of the forest $F$, is determined from (\ref{wei}), since $F'$ is a subgraph of the graph $\it\Psi^\aleph$. Therefore 

\begin{equation}
f'_{\tt XY}=\psi_{\tt XY}^\aleph=\lambda_{\tt XY} -\lambda_{\tt X}^\bullet \ . 
\label{weif}
\end{equation}
At the same time,
\begin{equation}
\lambda_{\tt XY}(F)\geq \lambda_{\tt XY}\ , \ \  \lambda^\bullet_{\tt X}(F)\geq \lambda^\bullet_{\tt X} \ .
\label{leq}
\end{equation}
The formula (\ref{weif}) includes the difference between the quantities $\lambda_{\tt XY}$ and $\lambda^\bullet_{\tt X}$. Therefore, the weight $\psi^\aleph_{\tt XY}$ can be either greater or less than the quantity $\psi_{xy}$ and the intuitively desired equality 

\begin{equation}
\psi^\aleph_{\tt XY}=\psi_{xy} 
\label{eq}
\end{equation}
is not satisfied (see Fig.\ref{p1}-\ref{p2}).

The proposed method of splitting preserves the tree structure of forests in the sense of Propositions 4 and 5. However, the weight of any arc $(\tt X,Y)$ in forests $F^\aleph$ and $F'$ (and the latter is a subgraph of the splitting directed graph $\it\Psi^\aleph$, and it is also a representative of forest $F$), appearing in these Propositions, is not the same. For an unweighted digraph $\it\Psi$, the forests $F^\aleph$ and $F'$ can be identified with each other since they have the same arcs. But if $\it\Psi$ is a weighted digraph, then the weight function of these forests is different.

\subsection{Minimum principals}

Let $F'\in{\cal F}^{k}(\it\Psi|\aleph)$, and let $F\in{\cal F}^{k}$ be its principal. And let forest $F$ have the minimum weight among all principals of forest $F'$. With respect to forest $F'$, forest $F$ is the minimal representable (or principal). We will call it the minimal principal.

\begin{theorem}
Let $F'\in{\cal F}^{k}(\it\Psi|\aleph)$ and $F$ be some principal of it. Then the forest $F$ is a minimal principal if and only if

\begin{equation}
\it\Upsilon^F_{\tt Z}= \lambda^\bullet_{\tt Z}, \  {\tt Z}\in\aleph_F^\bullet; \ \  \Upsilon^F_{\tt X} =\lambda_{\tt XY} ,  \ {\tt (X,Y)}\in{\tt A}F'. 
\label{l=}
\end{equation}
\end{theorem}

\begin{proof}
There are exactly $k$ sets of partition $\aleph$ labeled by forest $F$: $|\aleph_F^\bullet|=k$. For any of them, $\tilde{\cal T}^\bullet_{\tt Z}(F)={\cal T}^\bullet_{\tt Z}(F)=\{ F|_{\tt Z}\}$. Let $T\in\tilde{\cal T}^\bullet_{\tt Z}$. By Property 3, the graph $G=F^T_{\uparrow{\tt Z}}$ is a forest. Obviously, $G\in{\cal F}^{*k}$ and its representative is $F'$. By the hypothesis of the theorem,   

\begin{equation*}
0\leq \it\Upsilon^G-\it\Upsilon^F=\it\Upsilon^G_{\tt Z}-\it\Upsilon^F_{\tt Z}=\lambda^\bullet_{\tt Z}- \it\Upsilon^F_{\tt Z} \ ,
\end{equation*}
 from where $\it\Upsilon^F_{\tt Z}=\lambda^\bullet_{\tt Z}$. In this case $\it\Upsilon^F_{\tt Z} =\lambda^\bullet_{\tt Z}(F)$. 

Now let $T\in\tilde{\cal T}_{\tt XY}$. In the forest $F$, a single arc comes out from the set ${\tt X}$ (namely, from the vertex $x$) and enters the set ${\tt Y}$. In the tree $T$, the situation is exactly the same. Although there, generally speaking, the arc comes out of some other vertex $x'\in{\tt X}$, whose entry belongs to the set ${\tt Y}$. Therefore, the graph $H=F^T_{\uparrow{\tt X}}$ is still a forest and $H\in{\cal F}^{*k}$ and its representative is $F'$. By the condition of the theorem,

\begin{equation*}
0\leq \it\Upsilon^H-\it\Upsilon^F=\it\Upsilon^H_{\tt X}-\it\Upsilon^F_{\tt X}=\lambda_{\tt XY}- \it\Upsilon^F_{\tt X} \ ,
\end{equation*}
whence $\it\Upsilon^F_{\tt X}=\lambda_{\tt XY}$. In this case, $\it\Upsilon^F_{\tt X} = \lambda_{\tt XY}(F)$, since $\tilde{\cal T}_{\tt XY}(F)={\cal T}_{\tt XY}(F)=\{ F|_{\uparrow{\tt Z}}\}$. 
 
%And finally, the generated tree $F|_{\tt X}$ has the vertex $x$ as its root. Therefore $\{ F|_{\tt X}\}={\cal T}^{\bullet x}_{\tt X}(F)=\tilde{\cal T}^{\bullet x}_{\tt X}(F)= \tilde{\cal T}^{\bullet }_{\tt X}(F)$. This means that $\lambda^\bullet_{\tt X}(F)= \lambda^{\bullet x}_{\tt X}(F) $.

%Similar to the previous one, we are convinced that $\lambda^{\bullet x}_{\tt X}(F)= \lambda^{\bullet x}_{\tt X}$.   
 \end{proof}
 
 Note that given (\ref{lxy}), from the right-hand side (\ref{l=}) it follows that $\Upsilon^{F|_{\tt X}}= \lambda^{\bullet x}_{\tt X}$, where $x$ is the root of the generated tree $F|_{\tt X}$.
 
Note also that there remains an ambiguity --- the forest $F'$ may have several minimal principals. This is because, according to Proposition 5, when constructing the minimal principal in (\ref{ff'}), there remains arbitrariness in the choice of trees, now from the sets $\tilde{\cal T}_{\tt XY}$ and $\tilde{\cal T}^\bullet_{\tt Z}$.

Moreover, the transition to minimal principals does not save in the sense that equality (\ref{eq}) cannot be achieved. For arcs outgoing in the forest $F$ from elements of the set $\aleph^\circ_F$, it does not matter at all what happens to the arcs in the elements of $\aleph^\bullet_F$. And for any ${\tt X}\in\aleph^\circ_F$ the inequality still holds: $\lambda^\bullet_{\tt X}(F)= \lambda^{\bullet x}_{\tt X}\geq \lambda^{\bullet }_{\tt X}$. But now $\lambda_{\tt XY}(F) =\lambda_{\tt XY}$ and the following is true. 

\begin{theorem}
Let $F'\in{\cal F}^k(\it\Psi|\aleph)$, $F$ be its minimal principal, ${\tt X}\in\aleph_F^\circ$, 
and $(x,y)$ be the unique arc that comes from ${\tt X}$ in the forest of $F$, and let $y$ belong to the set ${\tt Y}\in\aleph$. Then

\begin{equation}
{\it f}'_{\tt XY}=\psi^\aleph_{\tt XY}\geq \psi_{xy} =f^\aleph_{\tt XY} \ . 
\label{vv}
\end{equation}
Equality in (\ref{vv}) is achieved when $ F|_{\tt X}\in\tilde{\cal T}^\bullet_{\tt X}$.
\end{theorem}

\begin{proof}
The forest $F'$ is a subgraph of the partition graph $\Psi^\aleph$, so the weights of its arcs $({\tt X,Y})$ are the same as those of $\Psi^\aleph$, that is, $f'_{\tt XY}=\psi_{\tt XY}^\aleph $. By Proposition 6, the right equality in (\ref{vv}) holds.

Let us compare the weight $\psi^\aleph_{\tt XY}$ with the weight $\psi_{xy}$. From the definition of the splitting weights (\ref{wei}) and its variant for $F^\aleph $ (\ref{weifa}) we obtain

\begin{equation}
 \psi^\aleph_{\tt XY}-\psi_{xy}= (\lambda_{\tt XY} -\lambda_{\tt X}^\bullet)- f_{\tt XY}^\aleph =(\lambda_{\tt XY} -\lambda_{\tt X}^\bullet)- (\lambda_{\tt XY}(F) -\lambda_{\tt X}^{\bullet}(F) ) \ . 
 \label{lll}
\end{equation} 
By Theorem 1, $\lambda_{\tt XY}(F) =\lambda_{\tt XY}$ and $\lambda^\bullet_{\tt X}(F) =\lambda^{\bullet x}_{\tt X}$. Therefore, (\ref{lll}) becomes

\begin{equation}
 \psi^\aleph_{\tt XY}-\psi_{xy} = \lambda_{\tt X}^{\bullet x} -\lambda_{\tt X}^\bullet \geq 0 \ . 
 \label{ll}
\end{equation}
The vertex $x$ is the root of the tree $F|_{\tt X}\in\tilde{\cal T}^{\bullet x}_{\tt X}$. 
We obtain the equality to zero (\ref{ll}) in a specific situation when $\lambda_{\tt X}^{\bullet x} =\lambda_{\tt X}^\bullet$, that is, when $\lambda_{\tt X}^{\bullet }(F) =\lambda_{\tt X}^\bullet$. And this means that $ F|_{\tt X}\in\tilde{\cal T}^\bullet_{\tt X}$. In this case, the set ${\tt X}$ does not contain the roots of the forest $F$. 
\end{proof}

Note that the forest $F$ from Theorems 1,2 is not minimal in the sense of the original definition. It is not even minimal among divisibles. It is simply a minimal principal of some representative $F'\in{\cal F}^k(\it\Psi|\aleph)$. And the latter in turn is simply a spanning subgraph of the splitting  digraph $\it\Psi$. It is not minimal among subgraphs with the same number of roots. But the values $\lambda_{\tt XY}$, $\lambda^{\bullet x}_{\tt X}$, $ \lambda^\bullet_{\tt Z}$ are defined as minima among all possible trees of the original graph $\it\Psi$. Therefore, if the sets ${\tt X}$, ${\tt Y}$, ${\tt Z}$ were arbitrary subsets of the set of all vertices ${\tt N}$, then we would not obtain any equalities (\ref{l=}). And since the original graph itself is a divisible by partition  $\aleph$ and the indicated subsets are elements of this partition, then the equalities (\ref{l=}) are absolutely natural and should be expected.   

For minimal principals, left inequalities (\ref{leq}) turn into equalities. But right inequalities remain.

\subsection{Minimum among divisibles}

We can introduce a subset ${\cal F^*}^k$ of the set ${\cal F}^k$ whose forests are tree-divisible: $F\in {\cal F^*}^k$ $\Leftrightarrow$ $F\in{\cal F}^k$ and $F$ are tree-divisible. And now, in order to match the weights of the arcs of the splitting graphs $\it\Psi^\aleph$ and $F^\aleph$, we can select from this set of forests a subset $\tilde{\cal F}^{*k}$ of forests that have the minimum possible weight among them: 
$F\in \tilde{\cal F}^{*k}$ $\Leftrightarrow$ 
$F \in {\cal F}^{*k}$ and $\it\Upsilon^F=\underset{G\in{\cal F}^{*k}}{\min}\it\Upsilon^{G}$. 

Note that the set $\tilde{\cal F}^{*k}$ (minimal among divisibles) does not coincide with the set of divisible minimal forests (those forests from $\tilde{\cal F}^k$ that are divisible). The set $\tilde{\cal F}^k$ may not contain any divisible forests at all, and at the same time $\tilde{\cal F}^{*k} \neq \emptyset$ may hold. In this case, only some forests of ${\cal F}^k\setminus \tilde{\cal F}^k $ are divisible. In the example in Fig. \ref{p1}, the minimum-weight spanning tree is not divisible by partition $\aleph$.

\begin{theorem}
Let $F'\in{\cal F}^{k}(\it\Psi|\aleph)$ and $F$ be its minimal principal. Then $F\in\tilde{\cal F}^{*k}$ if and only if $F'\in\tilde{\cal F}^{k}(\it\Psi|\aleph)$.
  \end{theorem}
\begin{proof} 

For a minimal principal $F$ of forest $F'$, using formula (\ref{l=}) from Theorem 1, we have:

\begin{equation*}
\it\Upsilon^F=\sum_{\tt Y\in\aleph_F}{\it\Upsilon^F_{\tt Y}} =\sum_{\tt X\in\aleph^\circ_F}{\it\Upsilon^F_{\tt X}}+\sum_{\tt Z\in\aleph^\bullet_F}{\it\Upsilon^F_{\tt Z}}=  \sum_{(\tt X,Y)\in{\tt A}F'}\lambda_{\tt XY}+ \sum_{{\tt Z}\in\aleph^\bullet_{F}} \lambda^\bullet_{\tt Z}.
\end{equation*}
Let us also write out the weight of the forest $F'$, the weights of the arcs of which are determined from (\ref{wei}): 

\begin{equation}
\it\Upsilon^{F'}= \sum_{(\tt X,Y)\in{\tt A}F'} (\lambda_{\tt XY}- \lambda^\bullet_{\tt X})=\sum_{(\tt X,Y)\in{\tt A}F'} \lambda_{\tt XY} -\sum_{{\tt X}\in\aleph^\circ_F} \lambda^\bullet_{\tt X} \pm \sum_{{\tt Z}\in\aleph^\bullet_F} \lambda^\bullet_{\tt Z} = \it\Upsilon^F- \sum_{{\tt Y}\in\aleph} \lambda^\bullet_{\tt Y} \ .
\label{yps}
\end{equation}
Note that the last sum in (\ref{yps}) is a fixed value.

Now let $G'$ be some forest of ${\cal F}^{k}(\it\Psi|\aleph)$. Generally speaking, it has a different set of arcs than $F'$. Nevertheless, the weight of any principal $G$ can still be represented as
\begin{equation*}
\it\Upsilon^G=\sum_{\tt Y\in\aleph}\it\Upsilon^G_{\tt Y} =\sum_{\tt X\in\aleph^\circ_G}{\it\Upsilon^G_{\tt X}}+\sum_{\tt Z\in\aleph^\bullet_G}{\it\Upsilon^G_{\tt Z}} .
\end{equation*}
f $G$ is a minimal principal, then using the formulas (\ref{l=}), the last equality can be rewritten as 

\begin{equation*}
\it\Upsilon^G= \sum_{(\tt X,Y)\in{\tt A}G'}\lambda_{\tt XY}+ \sum_{{\tt Z}\in\aleph^\bullet_{G}} \lambda^\bullet_{\tt Z}.
\end{equation*} 
In the same way as we determined the weight of $F'$, for $G'$ we obtain

\begin{equation}
\it\Upsilon^{G'}= \it\Upsilon^G- \sum_{{\tt Y}\in\aleph} \lambda^\bullet_{\tt Y} \ .
\label{ypsg}
\end{equation}
Comparing (\ref{yps}) and (\ref{ypsg}) we find that $F$ and $F'$ are minimal (each in its own class) or not simultaneously.
\end{proof}

Note that the transition from minimal principals to minimal among divisibles does not change the situation with the possibility of satisfying equality (\ref{eq}). Indeed, for any ${\tt X}\in\aleph^\circ_F$, on the one hand $F|_{\tt X}=\lambda^\bullet_{\tt X}(F)= \lambda^{\bullet x}_{\tt X}$, where $x$ is the root of the generated tree $F|_{\tt X}$. On the other hand $\lambda^{\bullet x}_{\tt X}\geq \lambda^{\bullet }_{\tt X}$. The value $\lambda^{\bullet }_{\tt X}= \lambda^{\bullet q}_{\tt X}$, where $q$ is some vertex from the set ${\tt X}$. And the vertex $q$ does not have to coincide with $x$. Therefore (\ref{ll}) remains an inequality.

The situation of equality $\psi^\aleph_{\tt XY}= \psi_{xy}$ (aka equality to zero (\ref{ll})), in particular, is realized when the weights of the original graph form a symmetric matrix. That is, in fact, when we are dealing with an undirected original graph. We will consider this issue in more detail in the section on undirected graphs.

\subsection{Removing irrelevant arcs} 

\begin{figure}[h]
\unitlength=0.7mm
\begin{center}
\begin{picture}(195,50)

\put(24,45){$\hat{\it\Psi}$}

\put(28,21){\scriptsize $ u$}
\put(27,2){\scriptsize $ v$}
\put(3,2){\scriptsize $ y$}
\put(48,20){\scriptsize $ r$}
\put(47,2){\scriptsize $ q$}
\put(49,38){\scriptsize $ p$}
\put(2,21){\scriptsize $ x$}
\put(2,38){\scriptsize $ s$}
\put(23,38){\scriptsize $ t$}

\put(8,6){\vector(1,0){17}}
\put(25,5){\vector(-1,0){17}}
\put(7,8){\vector(0,1){12}}
\put(6,20){\vector(0,-1){12}}

\put(8,21){\vector(1,0){17}}
\put(6,23){\vector(0,1){12}}
\put(45,20){\vector(-4,-3){17}}
\put(28,36){\vector(1,0){17}}
\put(46,35){\vector(0,-1){12}}
\put(46,8){\vector(0,1){12}}
\put(25,36){\vector(-4,-3){17}}
\put(26,23){\vector(0,1){12}}
\put(47,20){\vector(0,-1){12}}

\put(5,5){$\bullet$}
\put(5,20){$\bullet$}
\put(5,35){$\bullet$}
\put(25,5){$\bullet$}
\put(25,20){$\bullet$}
\put(25,35){$\bullet$}
\put(45,5){$\bullet$}
\put(45,20){$\bullet$}
\put(45,35){$\bullet$}

\put(17,1){\scriptsize $7$}
\put(17,7){\scriptsize $5$}
\put(15,22){\scriptsize $1$}
\put(14,30){\scriptsize $1$}
\put(7,27){\scriptsize $1$}
\put(8,13){\scriptsize $1$}
\put(2,13){\scriptsize $4$}

\put(35,31){\scriptsize $1$}
\put(33,15){\scriptsize $1$}
\put(43,12){\scriptsize $2$}
\put(43,27){\scriptsize $2$}
\put(23,27){\scriptsize $2$}
\put(48,12){\scriptsize $1$}

\put(82,45){$\hat{\it\Psi}|\aleph = \it\Psi|\aleph$}

\put(86,6){\oval(24,12)}
\put(86,29){\oval(27,25)}
\put(116,20){\oval(10,41)}

\put(78,13){\scriptsize $3$}
\put(93,13){\scriptsize $5$}
\put(105,23){\scriptsize $2$}
\put(104,14){\scriptsize $2$}

\put(114,19){\vector(-4,-3){17}}
\put(97,27){\vector(1,0){17}}
\put(92,20){\vector(0,-1){12}}
\put(81,8){\vector(0,1){12}} 

\put(85,27){${\tt X}$}
\put(85,4){${\tt Y}$}
\put(115,22){${\tt Z}$}

\put(148,45){$\{ T\}=\tilde{\cal F}^{*1}(\hat{\it\Psi})$}

\put(148,6){\vector(1,0){17}}
\put(146,35){\vector(0,-1){12}}
\put(168,36){\vector(1,0){17}}
\put(148,21){\vector(1,0){17}}
\put(166,23){\vector(0,1){12}}
\put(185,20){\vector(-4,-3){17}}
\put(186,8){\vector(0,1){12}}
\put(186,35){\vector(0,-1){12}}

\put(145,5){$\bullet$}
\put(145,20){$\bullet$}
\put(145,35){$\bullet$}
\put(165,5){$\bullet$}
\put(165,20){$\bullet$}
\put(165,35){$\bullet$}
\put(185,5){$\bullet$}
\put(185,20){$\bullet$}
\put(185,35){$\bullet$}
\put(156,6){\oval(24,12)}
\put(156,29){\oval(27,25)}
\put(186,21){\oval(10,41)}

\put(175,31){\scriptsize $1$}
\put(155,22){\scriptsize $1$}
\put(163,27){\scriptsize $2$}
\put(147,27){\scriptsize $1$}
\put(155,7){\scriptsize $5$}
\put(174,15){\scriptsize $1$}
\put(183,12){\scriptsize $2$}
\put(183,27){\scriptsize $2$}

\end{picture} 
\caption{\small 
The splittings of the original graph and the lightweight graph coincide. The minimum of the divisible spanning trees $T$ in the lightweight graph $\hat{\it\Psi}$ is the same as in the original. The forest $F$ from Fig. \ref{p2} is a spanning forest of the lightweight graph $\hat{\it\Psi}$, while the forest $H$, which has the same representative, is not a subgraph of $\hat{\it\Psi}$.}
\label{p3}
\end{center}
\end{figure}
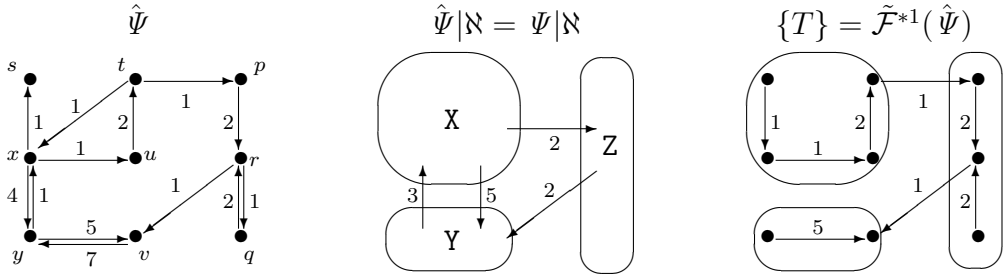

The weights of the splitting graph $\it\Psi|\aleph$ are calculated using the formula (\ref{wei}). We leave in $\it\Psi$ only those arcs that are present in the trees from $\tilde{\cal T}_{\tt XY}$ for all ordered pairs $(\tt X,Y)$ of atoms, and in trees from $\tilde{\cal T}^\bullet_{\tt X}$ for all ${\tt X}\in\aleph$. 
As a result, we obtain some lightweight graph $\hat{\it\Psi}$ such that $\hat{\it\Psi}|\aleph=\it\Psi|\aleph$. Let $F'\in{\cal F}^k(\hat{\it\Psi}|\aleph)$ for admissible $k$. Its principals belong to the set ${\cal F}^k(\hat{\it\Psi})$. But these principals are not minimal among the divisibles. Moreover, alas, not all of them are minimal principals. In this sense, the graph $\hat{\it\Psi}$ is not a graph of minimal principals, and certainly not a graph of minimal divisibles. The spanning forest $F$ from Fig. \ref{p2} is also a spanning forest of the lightweight graph (see Fig. \ref{p3}), but is not minimal among divisible ones. In turn, its representative $F'$ belongs to the set ${\cal F}^2(\hat{\it\Psi}|\aleph)$, but is not the minimal two-component in the class of representatives. On the contrary, it is the "heaviest" of these forests (see Fig. \ref{p1}). And the minimal one is a spanning forest consisting of one arc $(\tt X,Z)$ (or arc $(\tt Z,Y)$).

However, for any forest $F\in\tilde{\cal F}^{*k}$ nothing has changed (the minimal among the divisibles, in particular, is someone's minimal principal), but now $F\in\tilde{\cal F}^{*k}(\hat{\it\Psi})$. 
By Proposition 4 and Theorem 3, it now corresponds to exactly one representative $F'$ in the set $\tilde{\cal F}^k(\hat{\it\Psi}|\aleph)$. And the forest $F'$ is the same as the representative of the forest $F$ in the set $\tilde{\cal F}^k(\it\Psi|\aleph)$. And the weights of its arcs are exactly the same.

The very fact that some arcs can be removed from the original graph $\it\Psi$ without violating the tree structure corresponding to the partition $\aleph$ and preserving all minimal principals (including the minimal ones among the divisible ones) is, although quite natural, significant. Fig.\ref{p3} shows a lightweight graph $\hat{\it\Psi}$ that has the same splitting as the original one and, in particular, has the same minimal forest principals from $\it\Psi|\aleph$. It has a third fewer arcs.

\section{Undirected source graph}

From this point onwards, the terms graph, forest, tree refer specifically to undirected graphs, unless the opposite is clear from the context. For directed graphs, the exact designation is used --- directed graph, entering forest, entering tree.  

\subsection{Definitions for undirected graph}
 For undirected graphs, we denote the set of vertices in the same way as for directed ones, and we denote the set of edges of the graph {\sc P} by ${\tt E}{\textsc P}$.

A forest is an acyclic graph; a tree is a connected component of a forest, or simply a connected acyclic graph. 

In this section, the starting point is a weighted undirected graph $\Phi$ on the same vertex set ${\tt V}{\Phi}={\tt N}$. Its spanning forests and trees on subsets are considered.

Depending on the context, the pair $(i,j)$, for $i\neq j$, can be considered both as an arc of the directed graph $\it\Psi$ and as an edge of the graph $\Phi$. For the graph $ \Phi$, the pairs $(i,j)$ and $(j,i)$ define the same edge with weight $\textsc{\phi}_{ij}$.

For a subgraph {\sc G} of a graph $ \Phi$ and a subset of the vertex set ${\tt D}\subseteq {\tt N}$, we introduce standard weights  
\begin{equation}
\Upsilon^{\textsc G}_{\tt S}=\sum_{\begin{smallmatrix}i,j\in{\cal S} \\ (i,j)\in {\cal E}{\textsc G}\end{smallmatrix}} \textsc{\phi}_{ij} \ , \ \ \Upsilon^{\textsc G}=\Upsilon^{\textsc G}_{\tt N}=\sum_{(i,j)\in {\cal E}{\textsc G}} \textsc{\phi}_{ij} \ . \label{vesp}
\end{equation}
Here, in contrast to the directed situation (\ref{ves}), always $\Upsilon^{\textsc G}_{\tt D}=\Upsilon^{{\textsc G}|_{\tt D}} $, but there is no additivity property.  

The sets of spanning forests of the graph $\Phi$ consisting of $k$ trees are denoted by ${\sf F}^k$. Forests {\sc F} from ${\sf F}^k$ on which the minimum of weight $\Upsilon^{\textsc F}$ is achieved are called minimal. The set of such forests is denoted by $\tilde{\sf F}^k$. 

If ${\tt D}\subseteq {\tt N}$, then by ${\sf T}_{\tt D}$ we mean the set consisting of trees {\sc T} that are subgraphs of {\sf P} such that ${\cal V}{\textsc T}={\tt D}$. Trees ${\textsc T}\in {\sf T}_{\tt D}$ on which the minimum of weight $\Upsilon^{\textsc T}$ is achieved are called minimal on the set ${\tt D}$. The set of such trees is denoted by $\tilde{\sf T}_{\tt D}$, and the minimum of weight itself is denoted by $\nu_{\tt D}$: 

\begin{equation}
\nu_{\tt D}=\min_{{\textsc T}\in{\sf
T}_{\tt D}}\Upsilon^{\textsc T}  .  
\label{nu}
\end{equation}

Similarly to the digraphs, we introduce a set of trees ${\sf T}_{\tt XY}$ according to the rule: ${\textsc T}\in{\sf T}_{\tt XY}$ $\Leftrightarrow$ ${\textsc T}|_{\tt X}\in{\sf T}_{\tt X}$, $|{\tt V}{\textsc T}|=|{\tt X}|+1$ and ${\tt V}{\textsc T}\cap{\tt Y}\neq \emptyset$. Thus, in the tree {\sc T} there is a single edge, one endpoint of which belongs not to the set {\tt X}, but to the set {\tt Y}. The minimum weight of such a tree is denoted by $\nu_{\tt XY}$, and the set of trees on which it is achieved is denoted by $\tilde{\sf T}_{\tt XY}$: $\textsc T\in\tilde{\sf T}_{\tt XY}$ $\Leftrightarrow$ $\textsc T\in{\sf T}_{\tt XY}$ and $\Upsilon^{\textsc T}=\nu_{\tt XY} $.

We also define the set ${\sf T}_{\tt X}^\circ=\underset{{\tt Y}\in\aleph\setminus \{ {\tt X} \}}{\cup}{\sf T}_{\tt XY}$, the corresponding minimum weight $\nu^\circ_{\tt X}=\underset{{\textsc T}\in{\sf T}_{\tt X}^\circ}{\min}\Upsilon^{\textsc T}$ and the subset $\tilde{\sf T}_{\tt X}^\circ$ on which it is achieved.

\subsection{Splitting an undirected graph}

Similarly to the digraphs, tree divisibility is introduced by partition $\aleph$ and graph  $\Phi$ splitting by partition $\aleph$. Exactly. 

We will say that the graph $\Phi$ is {\it tree-divisible} by the partition $\aleph$, and the partition $\aleph$ itself is {\it tree-partition} with respect to the graph $\Phi$ if the subgraph $\Phi|_{\tt X}$ generated by any set ${\tt X}\in\aleph$ is connected (in this case the set ${\sf T}_{\tt X}\neq \emptyset$, which fully corresponds to the directed variant).  Otherwise, the graph splitting cannot be defined. {\it Splitting} $\Phi|\aleph=\Phi^\aleph$ (both parts of the equality are used to denote partition) of the graph $\Phi$ by partition $\aleph$ is a graph whose vertex set is the elements of partition $\aleph$: ${\tt V}\Phi^\aleph=\aleph$; an edge $({\tt X,Y})$, $\{\tt X,Y\}\subset\aleph$ belongs to the edge set of the graph $\Phi^\aleph$ if and only if ${\sf T}_{\tt XY}\neq \emptyset$. 

If the graph \Phi \ is reflexive (each vertex is incident to a loop), then the partition graph $\Phi^\aleph$ is considered reflexive by definition.

If $\Phi$ is a weighted graph, then the weights of the egges of splitting  are defined as follows:

\begin{equation}
{\textsc\phi}_{\tt XY}^\aleph=\nu_{\tt XY} -\nu_{\tt X} \ , \ \ {\sf T}_{\tt XY}\neq\emptyset \ .
\label{weip}
\end{equation}
If the set ${\sf T}_{\tt XY}=\emptyset$, then the edge ({\tt X,Y}) is not in the splitting graph.

It would seem that the definition of weights (\ref{weip}), introduced by analogy with directed graphs, can lead to an asymmetric weight matrix. But it is easy to verify that this is not the case. Indeed, it is obvious that

\begin{equation}
\nu_{\tt XY}=\nu_{\tt X} + \min_ {i\in{\tt X},  j\in{\tt Y}} {\textsc \phi}_{ij} \ .
\label{xy}
\end{equation}
Thus, the definition of weights (\ref{weip}) becomes a natural expression
 
\begin{equation}
{\textsc\phi}_{\tt XY}^\aleph= \min_ {i\in{\tt X},  j\in{\tt Y}} {\textsc \phi}_{ij} \  .
\label{weip'} 
\end{equation}
 
For a reflexive weighted graph $\Phi$, the weight of the loop ({\tt X,X}) is determined again based on the minimality

\begin{equation}
{\textsc\phi}_{\tt XX}^\aleph=\min_{i\in{\tt X}}{\textsc\phi}_{ii}  \ .
\label{loop}
\end{equation}

The convenience and naturalness of such a definition of loop weights is manifested in the case when the weight function of a directed graph (a directed graph of potential barriers) can be constructed from the weight function of a reflexive undirected graph (a potential graph) according to the rule \cite{V11}

\begin{equation}
\psi_{ij}=\textsc{\phi}_{ij}-\textsc{\phi}_{ii} \ . 
\end{equation}

\subsection{Viewing $\Phi$ as a digraph $\it\Phi$}

The undirected graph $\Phi$ can also be viewed as a directed graph, turning each edge into a pair of opposite arcs, thus creating a directed graph $\it\Phi$.

Between the original undirected graph $\Phi$ and the view of it as a directed graph $\it\Phi$, we have the following relations for the weights
\begin{equation}
\lambda_{\tt XY}(\it\Phi)=\nu_{\tt XY} \ , \ \ \lambda_{\tt X}^\circ(\it\Phi)=\nu_{\tt X}^\circ \ , \ \ 
\lambda_{\tt X}^\bullet(\it\Phi)=\lambda_{\tt X}^{\bullet x}(\it\Phi) =\nu_{\tt X} \ , \ \ x\in{\tt X} \ .
\end{equation}

Thus, for $\it\Phi$ and its spanning forests $F$, the Propositions and Theorems proved above are valid with $\it\Psi$ replaced by $\it\Phi$. But due to the symmetry of the weight function, there is an important feature. Let us consider it.

Propositions 1-6 remain the same, but Theorems 1 and 2 are strengthened. Now, for any ${\tt X}\in\aleph$, obviously $ \lambda_{\tt X}^\bullet(\it\Phi)=\lambda_{\tt X}^{\bullet x}(\it\Phi)$, where $x$ is an arbitrary vertex from ${\tt X}$. This means that in (\ref{ll}) an equality to zero is formed, and the inequality in (\ref{vv}) turns into an equality. In addition, for any spanning subgraph, all elements of the partition $\aleph$ are labeled, so the corresponding indexing disappears.

It is possible to make a complete reformulation of the proved Propositions and Theorems in terms of the undirected graph $\Phi$. We will not do this. We will limit ourselves to only reformulating the theorems.

{\bf Theorem 1'.} {\it
Let ${\textsc F}'\in{\sf F}^{k}(\Phi|\aleph)$ and ${\textsc F}$ be some principal of it. Then the forest $\textsc F$ is a minimal principal if and only if

\begin{equation}
\Upsilon^{\textsc F}_{\tt X}= \nu_{\tt X}, \ {\tt X}\in\aleph ; \ \  \mathrm{f}_{xy} =\nu_{\tt XY} ,  \ {\tt (X,Y)}\in{\tt E}{\textsc F}' .  
\end{equation}
Here 
$(x,y)$ is the only edge of the forest $\textsc F$ such that $x\in{\tt X}$ and $y\in{\tt Y}$.}

{\bf Theorem 2'.} {\it Let ${\textsc F}'\in{\sf F}^{k}(\Phi|\aleph)$ and ${\textsc F}$ be its minimal principal, $({\tt X,Y})\in{\tt E}{\textsc F}'$, 
and let $(x,y)$ be the unique edge of the forest ${\textsc F}$ such that $x\in{\tt X}$ and $y\in{\tt Y}$. Then

\begin{equation}
\mathrm f'_{\tt XY}=\textsc{\phi}^\aleph_{\tt XY}= \textsc{\phi}_{xy} =\mathrm f^\aleph_{\tt XY} \ . 
\label{ffp}
\end{equation}}
In particular, for an undirected graph $\Phi$, every forest 
$\textsc{F}'\in{\cal F}^k(\Phi|\aleph)$ can be identified with a 
splitting $\textsc{F}^\aleph$ of any of its minimal principals ${\textsc F}$.

{\bf Theorem 3'.} {\it 
Let $\textsc{F}'\in{\sf F}^{k}(\Phi|\aleph)$ and ${\textsc F}$ be its minimal principal. Then ${\textsc F}\in\tilde{\sf F}^{*k}$ if and only if  $\textsc{F}'\in\tilde{\sf F}^{k}(\Phi|\aleph)$. }

Given the expression (\ref{weip'}), the weights in (\ref{ffp}) look maximally natural and consistent with the weights of the arcs of the forests from $\tilde{\sf F}^{*k}$ and $\tilde{\sf F}^{k}(\Phi|\aleph)$ and the weights of the arcs of the original graph $\Phi$.

\section{Bottom line} 
\subsection{Divisibility problem}

It would seem that, by virtue of Theorems 2',3', undirected graphs do not have all the problems found in directed graphs. But, alas, this is not so. There remains one problem common to graphs and directed graphs. Let us voice it in terms of directed graphs.

As already pointed out, the forests $F\in \tilde{\cal F}^{*k}$ are not at all minimal in the sense of the original definition of minimality --- they are minimal among divisibles.

If we consider minimal forests $F\in\tilde{\cal F}^k$, then there is no reason to expect that $F$ will be tree-divisible. It is possible that none of the minimal forests is tree-divisible. At the same time, the set of forests ${\cal F}^k(\it\Psi^\aleph)$ is well defined and not empty, and forests with minimal weight can be identified among them. 

We "bypassed" this problem by first introducing a split of the original graph $\it\Psi$, and then began to consider minimal principals for its spanning subgraphs. But the issue still remains. In this case, we don't even have to focus on forests that are minimal in one sense or another. Let $F$ be a valid candidate to be divisible (by Proposition 2), i.e. $F\in{\cal F}^l$, $l\leq |\aleph|$. This forest may well not be divisible.

This serious circumstance, which does not allow an arbitrary partition $\aleph$ to be considered as some natural partition, can be looked at more generally. Let there be a coarser partition $\aleph'$, in the sense that $\aleph'$ is a subpartition of the partition $\aleph$. 
For a forest that is divisible by a partition $\aleph'$ (even if it is not the minimal among those divisible by that partition), there is no reason to expect that it will be divisible by $\aleph$.

This circumstance continues to be an obstacle both for the unweighted original directed graph $\it\Psi$ and even for the last "bastion" \ --- unweighted undirected graph $\Phi$. Moreover, the operation of removing inessential edges is useless for the latter. It has none.  That is, let $\Phi$ be an undirected unweighted graph divisible by a partition $\aleph$, $k\leq |\aleph|$. We create the graph $\Phi^*$, leaving in the original graph only those edges that occur in spanning forests from ${\sf F}^k(\Phi)$ divisible by this partition. But as is easy to see, all edges of the original unweighted graph are found in divisible forests. Therefore, $\Phi=\Phi^*$ and the above procedure does not change anything.

It is necessary to look for special types of partitions that are consistent with the weight function and the type of graph, as well as a special type of weight function so that the transition to sub-partitions does not cancel the divisibility property. And such partitions exist and there is a special type of weight function \cite{V11}, characteristic of many physical problems, when the work along a closed loop is equal to zero (the work against the field forces on a section of one arc $(i,j)$: $\it\Psi_{ij}-\Psi_{ji}$).

 \subsection{Which partitions to choose}

As we have seen, if the partition $ \aleph$ is arbitrary, then it certainly does not correspond to the structure of the weights of the original digraph $\it\Psi$ and, in particular, to its minimum spanning forests. Therefore, it is proposed to use as a partition the atoms of the algebras $\mathfrak{A}_k$ (these algebras, according to \cite[Theorem 1]{V7}, are nested), which are generated by minimal spanning forests consisting of $k$ trees. And here the question of whether the corresponding forests are divisible turns out to be extremely important. In \cite{V9} an important conjecture is proved that the forest $F\in\tilde{\cal F}^k\cup \tilde{\cal F}^{k-1}$, restricted to an element of the partition $\aleph_k$ (an atom of the algebra $\mathfrak{A}_k$), is a tree. Therefore, for such partitions there is no question of divisibility of minimal forests with the corresponding number of trees. All of them are divisible. Naturally, those algebras are chosen that differ from each other. The difference arises when in the convexity inequalities 
\cite{V6,V}

\begin{equation}
\varphi^{k-1}-\varphi^k\ge\varphi^k- \varphi^{k+1}, \ \ \varphi^k=\min_{F\in{\cal F}^k}{\it\Upsilon}^F,  
\label{convex}
\end{equation}
for different $k$ the strict inequality sign \cite{V7} appears.

\subsection{Where are these splits important and why entering forests}
 
In small diffusion (aka diffusion in strong fields up to a normalization factor), the labeled atoms ${\tt Z}\in\aleph_k^\bullet$ form sublimit distributions that are invariant on time scales \cite{V1,V2,VF} corresponding to different increments in (\ref{convex}).

If we are interested in individual trajectories of a random process, it is important to know where the trajectory can go. In this sense, analysis in classical outgoing forests is natural. All because they answer the question of where we can get from here and how quickly.  All the results obtained in this paper and in the works we refer to can be rewritten in terms of outgoing forests and trees. 

If, as usual in physics, we are interested in where we are mainly located, then it becomes important where the trajectories are attracted to. That is, it is important from where we can get here. And what are these places where everything converges. And this question naturally leads to the entering forests.

Returning to the random process, it turns out that entering forests are natural for the evolution of one-dimensional distributions.  The corresponding limit projectors with respect to the small diffusion parameter are extracted from the formulas for the coefficients of the eigenvectors of the diffusion generator (which corresponds to the Laplace matrix of the directed graph $\it\Psi$), which have the form of an unsigned sum over spanning forests \cite{V3}-\cite{V5}. In the asymptotics, only the terms corresponding to the minimal forests characteristic of the chosen time scale remain. And in this time scale corresponding to some $k$, the generator matrix turns into the Laplace matrix of the directed graph $\it\Psi|\aleph_k$.

\subsection{What is known and where to go}

In addition to the results obtained in this paper, in \cite{V12} a complete study of the properties of the restriction of minimal forests to atoms ${\tt X}\in\aleph_k$ of algebras of subsets $\mathfrak{A}_k$ and the influence of trees from $\tilde{\cal T}^\circ_{\tt X}$ and $ \tilde{\cal T}^\bullet_{\tt X}$ on the appearance and growth of forests from $\tilde{\cal F}^k$ as $k$ changes is carried out. Minimal trees, rather than forests, are constructed on the atoms of $\tt X$ by virtue of the proven hypothesis that the forest $F\in\tilde{\cal F}^k\cup \tilde{\cal F}^{k-1}$, restricted to an element of the partition $\aleph_k$, is a tree \cite{V9}. For the directed graph of potential barriers (\cite{V11}) (and, in particular, for undirected graphs) the situation is even "favorable". Any minimal forest from $\tilde{\cal F}^k$ turns out to be divisible not only by partition $\aleph_k$, but also by any partition $\aleph_l$ for $l>k$. It remains to construct splitting of the original digraph $\it\Psi$ by partitions $\aleph_k$. This is the subject of the next paper.

%\newpage
%\begin{center}

%\newpage

\centerline{Abstract}

\begin{center}{Splitting a graph by a given partition of the set of vertices based on the minimum weight of the induced trees}
\end{center}

\centerline{Buslov V.A.}

\parbox[t]{12cm}
{\small A method for considering a weighted directed graph with an accuracy of up to a given partition of the set of vertices is proposed. The resulting digraph (the splitting graph) does not contain arcs inside each partition element, and the arcs between the partition atoms are calculated in a special way taking into account the arcs of the original directed graph inside the atoms. This accounting is based on minimal trees defined on atoms. A study was made of what information about the original  digraph is preserved in its splitting.
}
\vspace{0.5cm}

St. Petersburg State University, Faculty of Physics, Department of Computational Physics

198504 St. Petersburg, Old Peterhof, st. Ulyanovskaya, 3

Email: abvabv@bk.ru, v.buslov@spbu.ru

\begin{thebibliography}{20}
%\end{center}


\bibitem{V6} 
V.A. Buslov,   
  {\em Structure of Minimum-Weight Directed Forests: Related Forests and Convexity Inequalities } 
J. Math. Sci. (N.Y.),  
2020,  
{\bf 247}  
(1), 
pp, 383-393.


\bibitem{V7}
V. A. Buslov, 
{\em The Structure of Directed Forests of Minimal Weight: Algebra of Subsets},
J. Math. Sci. (N.Y.), 
2021,
{\bf 255},
pp. 1-16.

\bibitem{V8}
 V.A. Buslov,  
{\em Algorithm for sequential construction of spanning minimal directed forests},
Journal of Mathematical Sciences (N.Y.), 
September 2023,
{\bf 275} 
(2), 
pp. 117-129.

\bibitem{V3}  
V. A. Buslov,   
{\em On characteristic polynomial coefficients of the Laplace matrix of a weighted digraph and all minors theorem},    
J. Math. Sci. (N.Y.),  
2016,  
{\bf 212}(2) 
pp. 643-653. 

\bibitem{V4} 
V. A. Buslov 
{\em On characteristic polynomial and eigenvectors in terms of tree-like structure of the graph},
J. Math. Sci. (N.Y.), 2018,
{\bf 232}(1),pp. 6–20.

\bibitem{V5} 
V. A. Buslov 
{\em On the relationship between multiplicities of the matrix spectrum and signs of components of its eigenvector in a tree-like structure},
J. Math. Sci. (N.Y.),
2019,
{\bf 236}
(5),
pp. 477–489.

\bibitem{V1} 
 V. A. Buslov, K. A. Makarov, 
{\em Hierarchy of time scales in the case of weak diffusion},
Theoret. and Math. Phys.,
1988, 
{\bf 76}(2),
pp. 818-826.

\bibitem{V2} 
V. A. Buslov, K. A. Makarov, 
{\em Lifetimes and lower eigenvalues of an operator of small diffusion}, Math. Notes, 
 1992,
{\bf 51}
(1), 
pp. 14-21.

\bibitem{V9}
V.A. Buslov, 
{\em When a forest, narrowed to an atom of subset algebra, turns out to be a tree}, 
2025, 	
\href{https://doi.org/10.48550/arXiv.2501.19068}{	arXiv:2501.19068},
pp. 1-30.

\bibitem{V10}
V.A. Buslov,
{\em Algorithm for Constructing Related Spanning Directed Forests of Minimum Weight},
2025, 	
\href{https://doi.org/10.48550/arXiv.2502.05946}{	arXiv:2502.05946}, 
pp. 1-19.

\bibitem{V11}
V.A. Buslov,  
{\em Digraphs of potential barriers: properties of their tree structure and algorithm for constructing minimum spanning forests}, 
2025,  	
\href{https://doi.org/10.48550/arXiv.2504.14484}{ 	arXiv:2504.14484}, 
pp. 1-32. 

\bibitem{V12}
V. A. Buslov,  
{\em How Trees on Atoms of Subset Algebras Define Minimal Forests and Their Growth}, 
2025,  	
\href{https://doi.org/10.48550/arXiv.2506.17921}{ arXiv:2506.17921}
pp. 1-23


\bibitem{VF}
 A. D. Ventsel and M. I. Freidlin
{\em Fluctuations in Dynamic Systems Under Small Random Disturbance},  Moscow, 1979,p. 429.

\bibitem{V}
A. D. Ventsel
{\em On the asymptotic of eigenvalues of matrices with elements of order $\exp\{-V_{ij}/2\varepsilon^2\}$},
Reports of the USSR Academy of Sciences, 
1972,
{\bf 202}(2),  
pp. 263-266.




\end{thebibliography}
\end{document}